\documentclass[12pt,a4paper]{article}
\usepackage[top=2cm, bottom=3cm, left=2cm, right=2cm]{geometry}

\usepackage[dvips]{graphicx}
\usepackage{amsmath,amsfonts,amssymb,amsthm,color}
\usepackage{boxedminipage}

\usepackage[numbers]{natbib}

\newtheorem{theorem}{Theorem}

\newtheorem{corollary}{Corollary}
\newtheorem{remark}{Remark}

\newtheorem{example}{Example}

\usepackage{bm}
\bmdefine{\Bt}{t}
\bmdefine{\BX}{X}
\bmdefine{\BY}{Y}
\bmdefine{\BZ}{Z}
\bmdefine{\BB}{B}
\bmdefine{\BM}{M}
\bmdefine{\BD}{D}
\bmdefine{\Bi}{i}
\bmdefine{\Bj}{j}
\bmdefine{\Bk}{k}
\bmdefine{\Bx}{x}
\bmdefine{\By}{y}
\bmdefine{\Bz}{z}
\bmdefine{\Bv}{v}
\bmdefine{\Bw}{w}
\bmdefine{\Bn}{n}
\bmdefine{\Ba}{a}
\bmdefine{\Bb}{b}
\bmdefine{\Bc}{c}
\bmdefine{\Be}{e}
\bmdefine{\Bu}{u}
\bmdefine{\Bp}{p}
\bmdefine{\Bzero}{0}
\bmdefine{\Bone}{1}

\def\znl#1{{\cal C}_{#1}}

\def\2{{1 \over \,2\,}}

\def\hyperF#1#2{{}_{#1}\kern-.05emF_{#2}}

\def\pdop{\partial}

\def\psum#1{\kern-0.75ex\mathop{\ \hbox{$\displaystyle\sum_{#1}$}^*}}
\def\psumtxt{\mathop{\ \hbox{$\sum$}^*}}

\def\pd#1{\partial_{#1}}

\title{Distribution of ratio of two Wishart matrices and evaluation of cumulative probability by holonomic gradient method} 
\author{
Hiroki Hashiguchi\thanks{Department of Mathematical Information Science, Tokyo University of Science}, 
Nobuki Takayama\thanks{Department of Mathematics, Kobe University}  \ 
and Akimichi Takemura\thanks{Data Science Education and Research Center, Shiga University}
}
\date{October, 2016}

\begin{document}

\maketitle

\begin{abstract}
We study the distribution of the ratio of two central Wishart matrices with different covariance matrices.  We first derive the density function of a particular matrix form
of the ratio and show that its cumulative distribution function can be expressed
in terms of the hypergeometric function $\hyperF{2}{1}$ of a matrix argument.  
Then we apply the holonomic gradient method  for numerical evaluation of 
the hypergeometric function.
This approach enables  us to compute the power function of Roy's maximum root test
for testing the equality of two covariance matrices.
\end{abstract}


\section{Introduction}

Let $W_1$ and $W_2$ be two independent  Wishart matrices having the distribution 
$W_m(n_1, \Sigma_1)$ and $W_m(n_2, \Sigma_2)$, respectively, where 
$W_m(n,\Sigma)$ denotes the $m\times m$ Wishart distribution with $n$ degrees of freedom and the
covariance matrix $\Sigma$.  
We assume $n_1, n_2 \ge m$ and $\Sigma_1, \Sigma_2$ are positive definite.
For testing the equality of covariance matrices
\begin{equation}
\label{eq:testing-problem}
H_0 : \Sigma_1 = \Sigma_2 
\end{equation}
we usually use test statistics based on the roots of $W_1 W_2^{-1}$.
$W_1 W_2^{-1}$ (or some symmetric variant of $W_1 W_2^{-1}$) is often called the
$F$ matrix. In this paper we are particularly interested in the largest root
$l_1(W_1 W_2^{-1})$  of $W_1 W_2^{-1}$, which is
Roy's maximum root statistic for testing $H_0$.
It is a natural test statistic for testing against the one-sided alternative hypothesis 
\[
H_1 : \Sigma_1 \ge \Sigma_2,
\] 
where the inequality is in the sense of  Loewner order.
\citet{kuriki-1993as} studied the likelihood ratio statistic against
this one-sided alternative.  

In order to compute  the power function of Roy's maximum root test, we need to
evaluate the probability $P(l_1(W_1 W_2^{-1})\le x)$ for the general case
$\Sigma_1 \neq \Sigma_2$.
The event $l_1(W_1 W_2^{-1})\le x$ can be written as
\begin{equation}
\label{eq:acceptance-region}
l_1(W_1 W_2^{-1}) \le x   \ \  \Leftrightarrow \ \ W_1 \le x W_2 \  \  \Leftrightarrow \ \
W_2^{-1/2} W_1 W_2^{-1/2} \le x I_m,
\end{equation}
where we specifically take $W_2^{1/2}$ to be the unique {\em positive definite square root} of $W_2$.
Note that in considering the distribution of the roots of $W_1 W_2^{-1}$, we can assume
$\Sigma_1 = I_m$ without loss of generality, because the distribution of the roots of $W_1 W_2^{-1}$
depends only on the roots of $\Sigma_1 \Sigma_2^{-1}$.
Under this additional assumption $\Sigma_1=I_m$, we derive the density function of the matrix
\begin{equation}
\label{eq:U}
U=W_2^{-1/2} W_1 W_2^{-1/2}
\end{equation}
and its cumulative distribution function $P(U\le \Omega)$, which involves
the hypergeometric function $\hyperF{2}{1}$ of a matrix argument.  
Then by specifying $\Omega = xI_m$ we obtain $P(l_1(W_1 W_2^{-1}) \le x)$.

By expressing the  cumulative distribution function of the maximum root in terms of $\hyperF{2}{1}$,  
we can apply the
holonomic gradient method (HGM, see e.g.\ \citet{hgd}) to numerically evaluate $\hyperF{2}{1}$.
In \citet{hashiguchi-etal-1f1} we have already shown that HGM works very well for $\hyperF{1}{1}$, 
which appears in Roy's maximum root test for the one-sample problem $H_0: \Sigma=\Sigma_0$.
Hence this paper is continuation of \citet{hashiguchi-etal-1f1} and demonstrates that HGM works well also
for $\hyperF{2}{1}$ unless the parameter values are extreme.

The organization of this paper is as follows.
In Section \ref{sec:distribution} we derive the density function and the cumulative distribution function of
$W_2^{-1/2} W_1 W_2^{-1/2}$. 
In Section \ref{sec:relations-to-known-results} we discuss relations of our results to earlier results on the $F$ matrix.
In Section \ref{sec:hgm} we study HGM for $\hyperF{2}{1}$, based on the partial differential equation
of Muirhead (\citet{muirhead-1970}, \citet{muirhead-book}).
In Section \ref{sec:numerical} we present results on numerical experiments of HGM.
We end the paper with some discussion of open problems in Section \ref{sec:discussion}.


\section{Distribution of ratio of two Wishart matrices}
\label{sec:distribution}
In this section we derive results on the density and the cumulative distribution function of
$U$ in \eqref{eq:U}.

First we present the following theorem concerning the density of $U$.

\begin{theorem}
\label{thm:density}
Under the assumption of $\Sigma_1=I_m$, the density function of $U= W_2^{-1/2} W_1 W_2^{-1/2} \; \ge 0$ is given by 
\begin{equation}
\label{eq:U-density}
f(U)
=\frac{\Gamma_m(\frac{n_1 + n_2}{2})|\Sigma_2|^{n_1/2} }
{\Gamma_m(\frac{n_1}{2})  \Gamma_m(\frac{n_2}{2}) }
 |I+\Sigma_2 U|^{-(n_1 + n_2)/2}
|U|^{(n_1 - m-1)/2}.
\end{equation}
\end{theorem}

\begin{proof}
Consider the transformation
$
(W_1, W_2)  \rightarrow (U, W_2)
$
with the Jacobian
\[
dW_1  dW_2 = |W_2|^{(m+1)/2} d U d W_2 .
\]
The joint density of $(W_1, W_2)$ is given as
\[
\frac{1}{2^{mn_1} \Gamma_m(\frac{n_1}{2})}  |W_1|^{(n_1-m-1)/2} 
\exp\left(-\frac{1}{2} {\rm tr}  W_1\right)
\frac{1}{2^{mn_2} \Gamma_m(\frac{n_2}{2}) |\Sigma_2|^{n_2/2}} |W_2|^{(n_2-m-1)/2} 
\exp\left(-\frac{1}{2} {\rm tr} \Sigma_2^{-1} W_2\right) .
\]
Therefore, letting
\[
 C_1 = \frac{1}{2^{mn_1} \Gamma_m(\frac{n_1}{2})2^{mn_2} \Gamma_m(\frac{n_2}{2}) |\Sigma_2|^{n_2/2}},
\] 
the joint density of $(U, W_2)$ is 
\begin{align}
f(U,W_2)&=C_1 |W_2|^{(m+1)/2} \times |W_2^{1/2}U W_2^{1/2}|^{(n_1-m-1)/2}
\exp\left( - \frac{1}{2} {\rm tr} W_2^{1/2} U W_2^{1/2}\right) \nonumber \\ 
& \qquad \times  |W_2|^{(n_2-m-1)/2} 
\exp\left(-\frac{1}{2} {\rm tr} \Sigma_2^{-1} W_2\right)  \nonumber \\
&= C_1 |W_2|^{(n_1 + n_2-m-1)/2} |U|^{(n_1 - m-1)/2} 
\exp\left(-\frac{1}{2} {\rm tr} \left( W_2^{1/2} U W_2^{1/2} + \Sigma_2^{-1} W_2\right) \right) \nonumber \\
&= C_1 |W_2|^{(n_1 + n_2-m-1)/2} |U|^{(n_1 - m-1)/2} 
\exp\left(-\frac{1}{2} {\rm tr}  (U + \Sigma_2^{-1}) W_2) \right) .
\label{eq:u-w2-joint-density}
\end{align}
Integrating this with respect to $W_2$ gives
\eqref{eq:U-density}.
\end{proof}

\begin{remark}
\label{rem:symmetric-square-root}
In Theorem \ref{thm:density} the positive definite square root $W_2^{-1/2}$ is essential.   
Other square roots do not work because of non-commutativity.
\citet{herz-1955} uses the positive definite square root for deriving the
integral expression of $\hyperF{2}{1}$ from that of $\hyperF{1}{1}$ and the argument of 
Theorem \ref{thm:density} follows Herz's derivation.
\end{remark}

From Theorem \ref{thm:density} we have the following expression
for the cumulative distribution function of $U$, which involves
the hypergeometric function $\hyperF{2}{1}$ of a matrix argument.  

\begin{theorem}
\label{thm:cdf}
Under the same assumption above, the cumulative probability  $P(U\le \Omega)$ is given by
\begin{equation}
\label{eq:cdf}
P(U\le \Omega)=\frac{\Gamma_m(\frac{m+1}{2})\Gamma_m(\frac{n_1+n_2}{2})}
{\Gamma_m(\frac{n_1 + m+1}{2})\Gamma_m(\frac{n_2}{2})} |\Sigma_2 \Omega|^{n_1/2}
\hyperF{2}{1}\left( \frac{n_1}{2},\frac{n_1 + n_2}{2}; \frac{n_1 + m+1}{2};-\Sigma_2 \Omega\right) .
\end{equation}
\end{theorem}

\begin{proof}
Let $C_2 = \Gamma_m(\frac{n_1 + n_2}{2})|\Sigma_2|^{n_1/2}/(\Gamma_m(\frac{n_1}{2})  \Gamma_m(\frac{n_2}{2}))$. Then
\[
P(U\le \Omega)= C_2 \int_{0\le U\le \Omega}   |I_m +\Sigma_2 U|^{-(n_1 + n_2)/2}
|U|^{(n_1 - m-1)/2} dU.
\]
Let $\tilde U = \Omega^{-1/2} U \Omega^{-1/2}$.
Then $P(U\le \Omega) = P(\tilde U\le I_m)$.
The Jacobian of the transformation is  $dU = |\Omega|^{(m+1)/2} d \tilde U$.
Hence
\begin{align*}
P(U\le \Omega) &= C_2  |\Omega|^{n_1/2}  \int_{0\le U\le I_m} 
  |I_m+\Sigma_2 \Omega^{1/2}\tilde U \Omega^{1/2}|^{-(n_1 + n_2)/2}
|\tilde U|^{(n_1 - m-1)/2} d\tilde U \\
&= C_2  |\Omega|^{n_1/2}  \int_{0\le U\le I_m} 
  |I_m+\Omega^{1/2} \Sigma_2 \Omega^{1/2}\tilde U |^{-(n_1 + n_2)/2}
|\tilde U|^{(n_1 - m-1)/2} d\tilde U \\
&= C_2 |\Omega|^{n_1/2} 
\frac{\Gamma_m(\frac{n_1}{2}) \Gamma_m(\frac{m+1}{2})}{\Gamma_m(\frac{n_1+m+1}{2})}
\hyperF{2}{1}\left(  \frac{n_1}{2},\frac{n_1 + n_2}{2}; \frac{n_1 + m+1}{2};-\Omega^{1/2} \Sigma_2 \Omega^{1/2}\right) \\
&=C_2|\Omega|^{n_1/2} 
\frac{\Gamma_m(\frac{n_1}{2}) \Gamma_m(\frac{m+1}{2})}{\Gamma_m(\frac{n_1+m+1}{2})}
\hyperF{2}{1}\left( \frac{n_1}{2},\frac{n_1 + n_2}{2}; \frac{n_1 + m+1}{2};-\Sigma_2 \Omega\right) .
\end{align*}
The last equality holds from the fact that the roots of 
$\Omega^{1/2}\Sigma_2 \Omega^{1/2}$ is the same as the roots of
$\Sigma_2 \Omega^{1/2}\Omega^{1/2}=\Sigma_2 \Omega$, including multiplicities,
and the hypergeometric function only depends on the roots of the matrix argument.
Rewriting the constants yields \eqref{eq:cdf}.
\end{proof}

\begin{remark}
Unlike Theorem \ref{thm:density}, in the above proof, 
$\Omega^{1/2}$ can be any square root of $\Omega$, i.e., it
does not have to be the positive definite square root of $\Omega$.
\end{remark}

By setting $\Omega=x I_m$ we have the following corollary.  We state this 
corollary without assuming $\Sigma_1 = I_m$ for the purpose of easier reference.
\begin{corollary}
\label{cor:max-root-result}
Let $W_1$ and $W_2$ be two independent  Wishart matrices having the distribution 
$W_m(n_1, \Sigma_1)$ and $W_m(n_2, \Sigma_2)$, respectively.
The the probability $P(l_1(W_1 W_2^{-1})\le x)$ is expressed as
\begin{align}
P(l_1(W_1 W_2^{-1})\le x) &= 
\frac{\Gamma_m(\frac{m+1}{2})\Gamma_m(\frac{n_1+n_2}{2})}
{\Gamma_m(\frac{n_1 + m+1}{2})\Gamma_m(\frac{n_2}{2})} x^{mn_1/2}|\Sigma_2 \Sigma_1^{-1}|^{n_1/2}\nonumber \\
& \qquad
\cdot \hyperF{2}{1}\left( \frac{n_1}{2},\frac{n_1 + n_2}{2}; \frac{n_1 + m+1}{2};-  x \Sigma_1^{-1} \Sigma_2 
\right).
\label{eq:roys-maxroot-2sample}
\end{align}
\end{corollary}
\citet{chikuse-1977} already obtained the expression \eqref{eq:roys-maxroot-2sample}.

Kummer relations for $\hyperF{2}{1}$ 
(see e.g.\ \citet{james-1964ams}, \citet{muirhead-book})
are referred as 
\begin{align}
 \hyperF{2}{1}(a, b; c; X) &= | I_m - X |^{c - a - b} \hyperF{2}{1}(c-a-b, c-b; c; X)
\label{eq:Kummer-2F1-01}, \\ 
\hyperF{2}{1}(a, b; c; X) &= | I_m - X |^{-b} 
 \hyperF{2}{1}(c - a, b; c; - X (I_m - X)^{-1}), 
\label{eq:Kummer-2F1-02}
 \end{align}
and we apply \eqref{eq:Kummer-2F1-01} to \eqref{eq:roys-maxroot-2sample} to obtain
\begin{align}
\Pr(l_1(W_1 W_2^{-1}) <  x)  &=
\dfrac{\Gamma_m (\frac{n_1}{2} + \frac{n_2}{2}) \; \Gamma_m (\frac{m+1}{2})}{
\Gamma_m (\frac{n_2}{2}) \; \Gamma_m (\frac{n_1+ m+1}{2})
} 
x^{\frac{n_1 m}{2}} |\Sigma_2^{-1} \Sigma_1|^{\frac{n_2}{2}} | \Sigma_2^{-1} \Sigma_1 + x I_m|^{- \frac{n_1 + n_2}{2}} 
\nonumber \\
&
\qquad \cdot \hyperF{2}{1}\left(
 \frac{m+1}{2},  \frac{n_1+ n_2}{2};  \frac{n_1+ m+1}{2}; x (\Sigma_2^{-1} \Sigma_1 + x I_m)^{-1}
\right).
\label{eq:roys-maxroot-2sample02}
\end{align}
In addition, we also have
\begin{align}
\Pr(l_1(W_1 W_2^{-1}) <  x) 
 &= 
\dfrac{\Gamma_m (\frac{n_1}{2} + \frac{n_2}{2}) \; \Gamma_m (\frac{m+1}{2})}{
\Gamma_m (\frac{n_2}{2}) \; \Gamma_m (\frac{n_1+ m+1}{2})
} 
x^{\frac{n_1 m }{2}} |\Sigma_2^{-1} \Sigma_1 + x I_m|^{- \frac{n_1}{2}}
\nonumber \\ 
& \hspace{2em}
\cdot \hyperF{2}{1}\left(
 \frac{n_1}{2},  - \frac{n_2}{2} + \frac{m+1}{2};  \frac{n_1+ m+1}{2}; x (\Sigma_2^{-1} \Sigma_1 + x I_m)^{-1}
\right) 
\label{eq:roys-maxroot-2sample03}
\end{align}
by applying \eqref{eq:Kummer-2F1-02} to \eqref{eq:roys-maxroot-2sample02}.
If $r = \frac{n_2}{2} - \frac{m+1}{2}$ is a non-negative integer, 
we see that  \eqref{eq:roys-maxroot-2sample03} is terminated as a finite series.
Both \eqref{eq:roys-maxroot-2sample} and   \eqref{eq:roys-maxroot-2sample03}
are alternating series, while \eqref{eq:roys-maxroot-2sample02} is 
a series of nonnegative terms.  In Section \ref{sec:numerical} we use
\eqref{eq:roys-maxroot-2sample02} for numerical stability in
evaluating the initial value for HGM.

\citet{chikuse-1977} 
mentioned that the upper probability of the 
smallest root $l_m(W_1 W_2^{-1})$ can be obtained from \eqref{eq:roys-maxroot-2sample}
 by replacing $n_1$, $n_2$, $x$, $\Sigma_1$ and 
$\Sigma_2$ by $n_2$, $n_1$, $x^{-1}$, $\Sigma_2$ and $\Sigma_1$, respectively.
\begin{align}
\Pr(l_m(W_1 W_2^{-1}) \ge  x) &= \dfrac{\Gamma_m (\frac{m+1}{2}) \Gamma_m(\frac{n_1 + n_2}{2})
}{
\Gamma_m(\frac{n_1}{2}) \Gamma_m(\frac{n_2 + m + 1}{2})
} \left| \frac{1}{x} \Sigma_2^{-1} \Sigma_1 \right|^{\frac{n_2}{2}}
\nonumber 
\\
& \qquad 
\cdot \hyperF{2}{1} \left(
 \frac{n_2}{2}, \frac{n_1 + n_2}{2}; \frac{n_2 + m + 1}{2}; - \frac{1}{x} \Sigma_2^{-1} \Sigma_1
\right)
\label{eq:roys-minroot-2sample}
\end{align}
The equation \eqref{eq:roys-minroot-2sample}
 can be obtained form the integral of \eqref{eq:U-density}
because of 
\begin{align*}
 & \int_{U > \Omega} | U |^{\frac{n_1}{2} - \frac{m+1}{2}} | I_m + \Sigma_2 U | ^{-  \frac{n_1+n_2}{2}} d U
\\ 
 & \quad = \dfrac{\Gamma_m(\frac{n_2}{2})  \Gamma_m(\frac{m+1}{2}) 
 }{
 \Gamma_m(\frac{n_2 + m+1}{2}) 
 }
  | \Sigma_2  |^{- \frac{n_1+n_2}{2}} \; 
 | \Omega |^{-\frac{n_2}{2}}
 \hyperF{2}{1} \left(\frac{n_2}{2}, \frac{n_1+n_2}{2}; \frac{n_2}{2} + \frac{m+1}{2};  - (\Omega \Sigma_2)^{-1} \right),
\end{align*}
$\Pr(l_m(W_1 W_2^{-1})  > x ) = \Pr(U > x I_m)$ and 
the substitution of $\Sigma_2$ by $ \Sigma_1^{-1} \Sigma_2$.
The above integral can be found in  Problem 1.17 in p.50 of \citet{gupta-nagar-2000}.

\section{Relations to known results on the $F$ matrix and the beta matrix}
\label{sec:relations-to-known-results}
In this section we discuss the relationships between our results and earlier results.  We mainly consider
the null case $\Sigma_1 = \Sigma_2$.

\citet{constantine-1963} gave the cumulative distribution 
functions of the Wishart and the multivariate beta distributions.
  Let $B = (W_1 + W_2)^{-1/2} W_1  (W_1 + W_2)^{-1/2}$
where $W_1 \sim W_m(n_1, \Sigma)$, $W_2 \sim W_m(n_2, \Sigma)$, $n_1, n_2 \ge m$
and  $\Sigma > 0$.  Then the random matrix $B$ follows the multivariate beta distribution of 
the first kind with parameters $\frac{n_1}{2}$ and $\frac{n_2}{2}$.  Note that in the null case
$(W_1 + W_2)^{-1/2}$ does not have to be the positive definite square root, because of the orthogonal invariance of the beta distribution.
The cumulative distribution function of $B$ is given by
$$
 \Pr(B < \Omega) = \dfrac{\Gamma_m(\frac{n_1 + n_2}{2}) 
 \Gamma_m(\frac{m+1}{2})
 }{
 \Gamma_m(\frac{n_2}{2}) \Gamma_m(\frac{n_1+m+1}{2})
 }
 | \Omega |^{\frac{n_1}{2}} \hyperF{2}{1}\left(
  \frac{n_1}{2}; - \frac{n_2}{2} + \frac{m+1}{2}; \frac{n_1 + m+ 1}{2}; \Omega
 \right) $$
for $0 < \Omega < I_m$.
Therefore the cumulative distribution function of the 
largest root $b_1(B)$ of $B$ is given as 
$\Pr(b_1(B) \le x ) = \Pr(B \le x I_m)$ and,
from the relationship 
$$
\Pr (l_1(W_1 W_2^{-1}) \le x) = \Pr \left(b_1(B) \le  \frac{x}{1+x} \right), 
$$
we also have
\begin{align}
\Pr(l_1(W_1 W_2^{-1}) \le  x) &= \dfrac{\Gamma_m(\frac{n_1 + n_2}{2}) 
 \Gamma_m(\frac{m+1}{2})
 }{
 \Gamma_m(\frac{n_2}{2}) \Gamma_m(\frac{n_1+m+1}{2})
 }
\left(\frac{x}{1+x} \right)^{\frac{n_1 m}{2}} 
\nonumber 
\\
& \qquad \cdot \hyperF{2}{1}\left(
  \frac{n_1}{2}; - \frac{n_2}{2} + \frac{m+1}{2}; \frac{n_1 + m+ 1}{2}; 
  \frac{x}{1+x} I_m
 \right) 
 \label{eq:maxroot-null-constantine}
\end{align}
for $x \ge 0$.  The above equation is the same as \eqref{eq:roys-maxroot-2sample03} by substitutions of $\Sigma_1 = \Sigma_2$

Based on the results of \citet{khatri-1972jmva},
for the case that $r=\frac{n_2}{2} - \frac{m+1}{2}$ is a nonnegative integer, 
\citet{venables-1973jmva} 
obtained another expression of
$\Pr(b_1(B) < x)$ which is equivalent to 
\begin{align}
  \Pr \left(l_1(W_1 W_2^{-1}) \le x \right) = \left( \frac{x}{1+x} \right)^{\frac{n_1 m}{2}} 
  \sum_{k=0}^{r m} \frac{(1+x)^{- k}}{k !} \psum{\kappa \vdash k} (\frac{1}{2} n_1)_{\kappa} \; \znl{\kappa}(I_m) \label{eq:maxroot-null-venables}
\end{align}
where $r = \frac{n_1 -  m - 1}{2}$ is a positive integer and  $\psumtxt$ 
denotes the summation over all partitions $\kappa =(\kappa_1, \dots, \kappa_m)$ of $k$
with $\kappa_1 \le r$. 
For example, $m=3, n_1=6$ and $n_2 =10$, the equation \eqref{eq:maxroot-null-constantine} is 
\begin{align*}
& \frac{2145 x^9}{(x+1)^9}
 \left(-\frac{2 x^9}{143 (x+1)^9}+\frac{27 x^8}{143
   (x+1)^8}-\frac{166 x^7}{143 (x+1)^7}+\frac{9149 x^6}{2145 (x+1)^6}-\frac{113
   x^5}{11 (x+1)^5} 
   \right.
  \\
 &\left. \mbox{} +\frac{184 x^4}{11 (x+1)^4}-\frac{55 x^3}{3 (x+1)^3}+\frac{13
   x^2}{(x+1)^2}-\frac{27 x}{5 (x+1)}+1\right)
\end{align*}
and \eqref{eq:maxroot-null-venables} gives 
\begin{align*}
& 
\frac{x^9}{(x+1)^9} \left( \frac{30}{(x+1)^9}+
   \frac{135}{(x+1)^8} + \frac{330 }{(x+1)^7}+
   \frac{539}{(x+1)^6} 
 \right.
 \\
 &\mbox{}  + \left.
 \frac{ 531}{(x+1)^5}+
   \frac{360}{(x+1)^4} + \frac{165}{(x+1)^3}
   +\frac{45 }{(x+1)^2}+\frac{9 }{x+1}+1\right).
\end{align*}
We find that they are the same by their subtraction by symbolic computation.  
In general, it seems to be difficult to show that 
\eqref{eq:maxroot-null-constantine} and \eqref{eq:maxroot-null-venables} are equivalent, as pointed out in the concluding remarks of 
\citet{venables-1973jmva}. 

Finally we mention a result of \citet{khatri-1967ams} in the nonnull case. 
Section 3.4 of \citet{khatri-1967ams} gives  the density function of $l_1(W_1 W_2^{-1})$ 
in terms of $\hyperF{3}{2}$:
\begin{align}
f(l_1)&= c_2 |\Lambda|^{-\frac{n_1}{2}} l_1^{\frac{mn_1}{2}-1} |I_m+l_1 \Lambda^{-1}|^{-\frac{n_1+n_2}{2}}
\nonumber \\
& \qquad \cdot \hyperF{3}{2}\big(\frac{n_1+n_2}{2}, \frac{m}{2}+1, \frac{m-1}{2}; \frac{m}{2}, \frac{n_1+m+1}{2};
l_1 (\Lambda + l_1 I_m)^{-1}\big),
\label{eq:l1-density-by-khatri}
\end{align}
where $\Lambda = \Sigma_1 \Sigma_2^{-1}$ and
\[
c_2 = \frac{\Gamma(\frac{1}{2}) \Gamma_m(\frac{n_1+n_2}{2}) \Gamma_{m-1}( \frac{m}{2} +1)}
{\Gamma(\frac{m}{2})\Gamma(\frac{n_1}{2})\Gamma_m(\frac{n_2}{2})\Gamma_{m-1}(\frac{n_1+m+1}{2})}.
\]
Differentiation of \eqref{eq:Kummer-2F1-02} should yield
\eqref{eq:l1-density-by-khatri}, but it does not seem to be obvious.

\section{Holonomic gradient method for $\hyperF{2}{1}$}
\label{sec:hgm}

Applying the HGM for a numerical evaluation of the matrix hypergeometric
function ${}_2F_1$ is analogous to the case of ${}_1F_1$ of
\citet{hashiguchi-etal-1f1}.
We explain mainly the differences briefly.
Put
\begin{equation} \label{eq:2f1eq}
  g_i = \pd{i}^2 + [ p(x_i) + \sum_{j \not= i} q_2(x_i,x_j)] \pd{i}
  - \sum_{j \not= i} q(x_i,x_j) \pd{j} - r(x_i), \quad i=1,\dots,m,
\end{equation}
where
\allowdisplaybreaks
\begin{eqnarray*}
  p(x_i)&=& \frac{c-(m-1)/2-(a+b+1-(m-1)/2)x_i}{x_i (1-x_i)}, \\
  q_2(x_i,x_j)&=& \frac{1}{2(x_i-x_j)}, \\
  q (x_i,x_j)&=& \frac{x_j (1-x_j)}{2x_i (1-x_i) (x_i-x_j)}, \\
  r(x_i) &=& \frac{a b}{x_i(1-x_i)}.
\end{eqnarray*}
The matrix hypergeometric function ${}_2F_1(a,b,c;x_1, \ldots, x_m)$
is annihilated by the linear partial differential operator $g_i$'s
by \citet{muirhead-1970}.

\begin{theorem}
The set $\{ g_i \}_{i=1}^m$ is a Gr\"obner basis in the ring of 
differential operators with rational function coefficients
$R_m = {\bf C}(x_1, \ldots, x_m) \langle \pd{1}, \ldots, \pd{m} \rangle$.
\end{theorem}

\begin{proof}
Put $G_i = x_i(1-x_i) g_i$.
Then, we can see 
$$ [G_i, G_j] = \frac{1}{2}
 \frac{2 x_i x_j - x_i - x_j}{(x_i-x_j)^2} (G_i-G_j)
$$
by calculation.
Let us consider the graded reverse lexicographic order among the monomials of
$\pd{1}, \ldots, \pd{m}$.
Then, the leading term of $G_i$ is $x_i (1-x_i) \pd{i}^2$
and the leading term of $G_i$ and $G_j$ are coprime when $i \not= j$.
Therefore, the commutator $[G_i,G_j]$ can be regarded as an $S$-pair and the relation above
leads the $S$-pair criterion $[G_i,G_j] \longrightarrow 0$ by
$\{ G_i \}$.
Hence, $\{g_i\}$ is a Gr\"obner basis.
\end{proof}

Let $M$ be the left ideal of $R_m$ generated by $g_i$, $i=1, \ldots, m$.
The important conclusion of this theorem is that the system 
can be transformed into a completely integrable Pfaffian system
$$ \pd{i} F \equiv P_i(x) F \quad \mathop{\rm mod} M  $$
where $P_i(x)$ is a $2^m \times 2^m$ matrix and
$F$ is a column vector of length $2^m$ whose $i$-th entry 
($i = 0, 1, \ldots, 2^m-1$)
is
$ \pd{}^\alpha = \prod_{k=1}^{m} \pd{k}^{\alpha_k}$,
$\alpha=(\alpha_1, \alpha_2, \ldots, \alpha_{m})$.
Here, $\alpha$ is a vector obtained by the binary expansion of $i$ as
$ i = \sum_{k=0}^{m-1} \alpha_{k+1} 2^k$, $\alpha_k \in \{0, 1\}$.
We use this Pfaffian system for a numerical evaluation of ${}_2F_1$.

We derive an explicit expression of the matrix $P_i$ and 
we utilize a sparsity of the matrix $P_i$ in our implementation
as follows.
Put 
$[m] = \{1, 2, \ldots, m\}$.
Suppose that $J \subset [m]$.
The set $J$ is encoded into a binary number {\tt jj};
if $k \in J$, then the $k$-th bit\footnote{Here, 
we count $1$-th bit, $2$-th bit,
$\ldots$ instead of counting $0$-th bit, $1$-th bit, $\ldots$
of the convention of programming.}
of {\tt jj} is $1$ and
if $k \not\in J$, then the $k$-th bit of {\tt jj} is $0$.
Note that when {\tt jj} is the encoding of the set $J$,
the encoding of $\{k\} \cup J$ is \verb@ ((1 << (k-1)) | jj) @
in the language C.
The operator $\pd{J}$ is defined as
$\prod_{k \in J} \pd{k}$.
Note that any element of $F$ can be written in this form.
Let us describe the matrix $P_i$ in terms of $\pd{J}$.

Let $I$ be a subset of $[m]$.
When $i \not\in I$, we have $\pd{i} \pd{I} = \pd{I'}$ where
$I' = \{i\} \cup I$ and $\pd{I'}$ is again an element of $F$.
Then, the corresponding row of the matrix $P_i$ is  a unit vector.
Suppose $i \not\in J$ and
put $I = \{i \} \cup J$.
We want to express $\pd{i} \pd{I}$
in terms of the elements of $F$ modulo the left ideal $M$.
Apply $\pd{J}$ to the operator $g_i$ as
\begin{eqnarray*}
  \pd{J} g_i &=& \pd{i}^2 \pd{J} + p(x_i) \pd{i}\pd{J} \\
  & & + \pd{J} \sum_{j\not=i} q_2(x_i,x_j) \pd{i}
      - \pd{J} \sum_{j\not=i} q(x_i,x_j) \pd{j} \\
  & & - r(x_i) \pd{J}.
\end{eqnarray*}
When $k \not\in J$, we have
$  \pd{J} q_2(x_i,x_k) \pd{i} = 
   q_2(x_i,x_k) \pd{I}$, \\
When $k \in J$, we have
$  \pd{J} q_2(x_i,x_k) \pd{i} = 
   q_2(x_i,x_k) \pd{I} + \frac{\partial q_2(x_i,x_j)}
   {\partial x_k} \pd{I \setminus \{k\}}
$, \\
When $k \not\in J$, we have
$ \pd{J} q(x_i,x_k) \pd{k} =
  q(x_i,x_k) \pd{J \cup \{k\}} $,  \\
When $k \in J$, we have
$ \pd{J} q(x_i,x_k) \pd{k} =
  q(x_i,x_k) \pd{J \setminus \{k\}} \pd{k}^2 
 + \frac{\partial q(x_i,x_k)}{\partial x_k} \pd{J}
$. \\
In summary, when
$i \not\in J$, we have
\begin{eqnarray*}
&&\underline{\pd{i}^2 \pd{J}} 
 + p(x_i) \pd{I}
 + \sum_{k \not=i} q_2(x_i,x_k) \pd{I}
 + \sum_{k \not=i, k \in J} \frac{\partial q_2(x_i,x_k)}{\partial x_k} \pd{I \setminus \{k\}} \\
& &
 - \sum_{k \not=i, k \not\in J} q(x_i,x_k) \pd{J \cup \{k\}}
 - \sum_{k \not=i, k \in J} q(x_i,x_k) \underline{ \pd{J \setminus  \{k\}} \pd{k}^2}
 - \sum_{k \not=i, k \in J} \frac{\partial q(x_i,x_k)}{\partial x_k} \pd{J} \\
& & - r(x_i) \pd{J} \equiv 0
\end{eqnarray*}
modulo the left ideal $M$ generated by $g_i$, $i=1, \ldots, m$.
It follows from this relation that 
the operator $\pd{i}^2 \pd{J}$ can be expressed in terms of
the element of $F$ inductively with respect to $\sharp J$
(the cardinality of $J$).


\section{Numerical experiments}
\label{sec:numerical}

The numerical evaluation by HGM consists of two steps.
The first step is an approximate evaluation of  ${}_2F_1$ 
and its derivatives $\pd{I} \bullet {}_2 F_1$, $I \subset [m]$
by \citet{koev-edelman}
at an initial point  $x=x_0$ ($x_0$ is {\tt q0} in our package) 
which is close to the origin.
The second step is the application of the Runge-Kutta method
to the ordinary differential equation obtained from the Pfaffian system.

We find that the numerical evaluation of $\hyperF{2}{1}$ is more challenging
than $\hyperF{1}{1}$ in \citet{hashiguchi-etal-1f1}.
Hence we employ some heuristics for the numerical evaluation.
\begin{enumerate}
\item We use the formula (\ref{eq:roys-maxroot-2sample02})
for the step one,
because all terms of ${}_2F_1$ is non-negative when $x$ is positive.
Convergence of series whose entries have alternating signs 
as \eqref{eq:cdf} in Theorem \ref{thm:cdf} is slow in general.
\item For small positive number $x_0$, 
we evaluate ${}_2F_1$ at $x_0(\Sigma_2^{-1} \Sigma_1 + x_0 I_m)^{-1}$.
Let $f_k$ be the $k$-th approximation of this series, which is the truncation
of the series more than degree $k=|\lambda|$ terms where $\lambda$ is 
a partition.
When $|(f_k-f_{k-1})/f_{k-1}|$ is smaller than 
{\tt assigned\_series\_error}, of which default value is $10^{-5}$ 
in our implementation,
we use the $k$-th approximation as the value of the step  one.
When the value ${\rm Pr}(\ell_1<x_0)$ is smaller than the assigned value
{\tt x0value\_min}, of which default value is $10^{-60}$
in our implementation, we increase $x_0$
and retry the evaluation.
Too small initial value is not acceptable for the Runge-Kutta method
by the double precision.
\item We use the adaptive version of the Runge-Kutta method
by 
the default value of relative and absolute errors are 
$10^{-10}$ and $|f_k-f_{k-1}| \times \mathop{\rm Pr}(\ell_1<x_0)/f_k$
in our implementation.
\item When one of the entry of the initial evaluation point 
$x_0(\Sigma_2^{-1} \Sigma_1 + x_0 I_m)^{-1}$
is close to $1$, the convergence of series ${}_2F_1$ becomes very slow.
In this case, we should decrease $x_0$ and the absolute error 
for the Runge-Kutta method.
\end{enumerate}
Under the above heuristics, numerical evaluation works well 
unless the parameters are extreme, for example when  
$n_i$'s are large, or the ratio of eigenvalues of
$\Sigma_2^{-1} \Sigma_1$ is large.
Systematic experiments and studies on parameter spaces for which
the HGM works well and improved algorithms will be future research topics.

Let us present some numerical examples to illustrate some border cases of performance of HGM.
We use our implementation of
the package {\tt hgm}\footnote{The demonstration in this paper was
performed on the version 1.16. It is newer than the version on cran.
This version can be obtainable from
{\tt http://www.math.kobe-u.ac.jp/OpenXM/Math/hgm}
}
for the system {\tt R}. 
The command {\tt hgm.p2wishart} evaluates the cumulative distribution function
in (\ref{eq:roys-maxroot-2sample02}).
The arguments are {\tt m}=$m$ (the dimension),
${\tt beta}=\mbox{the eigenvalues of }\Sigma_2^{-1} \Sigma_1$,
${\tt n1}=n_1$, ${\tt n2}=n_2$ (the degrees of freedom of two Wishart distributions respectively), 
${\tt q}$ (the last point of the evaluation interval) 
and ${\tt qo}$ (the point for the initial value).

\begin{example} 
We evaluate $\mathop{\rm Pr}(\ell_1 < x)$ for
$m=10$, $n_1=11$, $n_2=12$, and
$$\Sigma_2^{-1} \Sigma_1 = \mathop{\rm diag}(1,2,3,4,5,6,7,8,9,10)$$
by our implementation.
The starting point $x_0$ is $0.1$ and the approximation by the zonal polynomial
expansion is truncated at the degree 28, which is automatically determined
by the heuristics above.
It takes 13 minutes and 15 seconds to obtain Figure 
\ref{fig:m10_11_12} on a machine with Intel Xeon CPU (2.70GHz) with 256 G memory.
\begin{figure}[tb]
\begin{center}
\includegraphics[width=10cm]{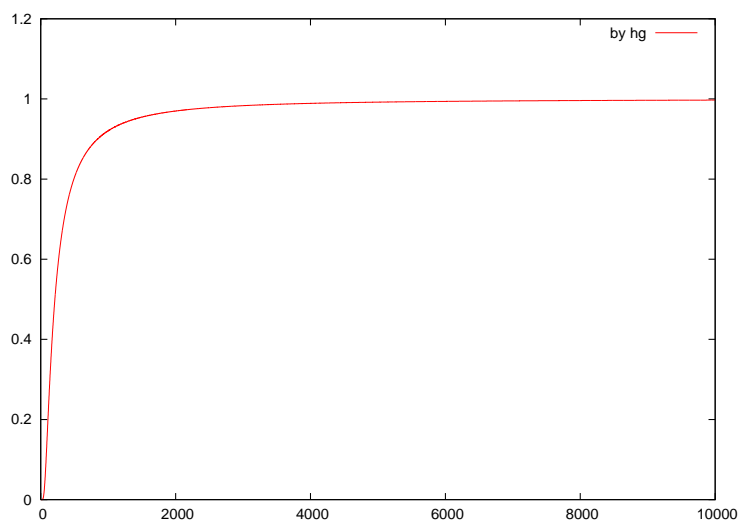}
\end{center}
\caption{$\mathop{\rm Pr}(\ell_1 < x)$, $m=10$}  \label{fig:m10_11_12}
\end{figure}
\end{example}

Let us see the behavior of our algorithm
when  the eigenvalues of
$\Sigma_2^{-1} \Sigma_1$
are a mixture of relatively small numbers and large numbers.
\begin{example}
The command
\begin{verbatim}
plot(hgm.p2wishart(m=3,beta=c(1,2,3),q=300,n1=10,n2=20,autoplot=1))
\end{verbatim}
works fine (no graph shown), 
but when we increase the eigenvalues $2$ and $3$ to $20$ and $300$ as
\begin{verbatim}
plot(hgm.p2wishart(m=3,beta=c(1,20,300),q=500,
                   n1=10,n2=20,autoplot=1),ylim=c(0,0.3))
\end{verbatim}
we get a warning ``abserr seems not to be small enough'',
which means that the control parameter for the absolute error for the adaptive Runge-Kutta method
is not small enough.
The default value of the absolute error is $1 \times 10^{-10}$.
The output is the left graph of Figure \ref{fig:1-20-300} and it looks like a wrong evaluation.
The trouble occurs when the initial value for the HGM is very small.
We should make the absolute error smaller or the initial evaluation point
${\tt q0}$, of which default value is $0.3$, larger.
Note that making {\tt q0} larger requires a lot of resources for approximate
evaluation of the series expansion of ${}_2F_1$.
Then, we retry the command with a new {\tt err} parameter vector, 
which specifies
the absolute error and the relative error for the adaptive Runge-Kutta method,
as
\begin{verbatim}
plot(hgm.p2wishart(m=3,beta=c(1,20,300),q=500,
                   n1=10,n2=20,err=c(1e-30,1e-10),autoplot=1))
\end{verbatim}
The output is the right graph of Figure \ref{fig:1-20-300} and looks a correct evaluation.
\begin{figure}[tb]
\begin{center}
\includegraphics[width=6cm]{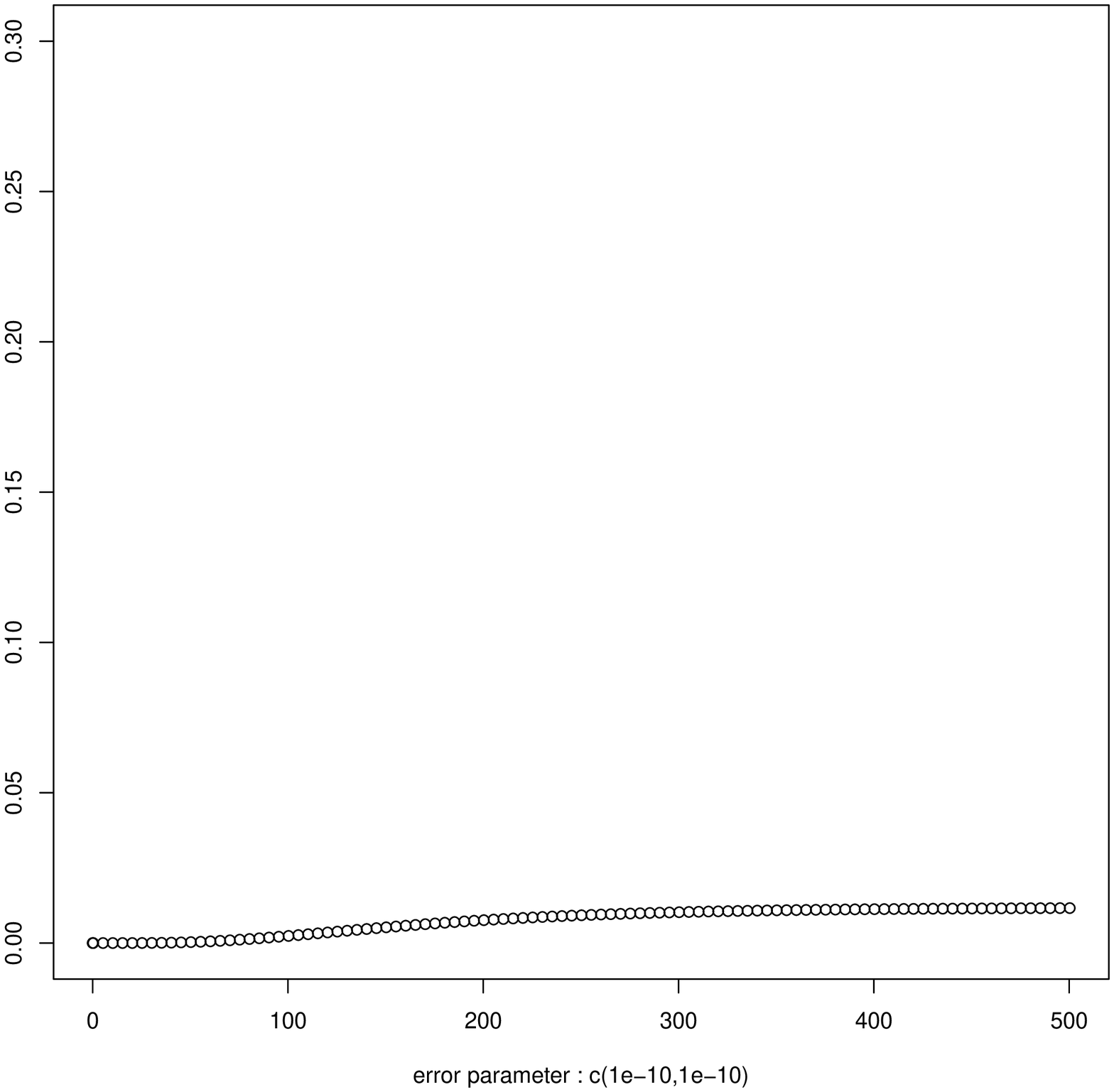} \quad
\includegraphics[width=6cm]{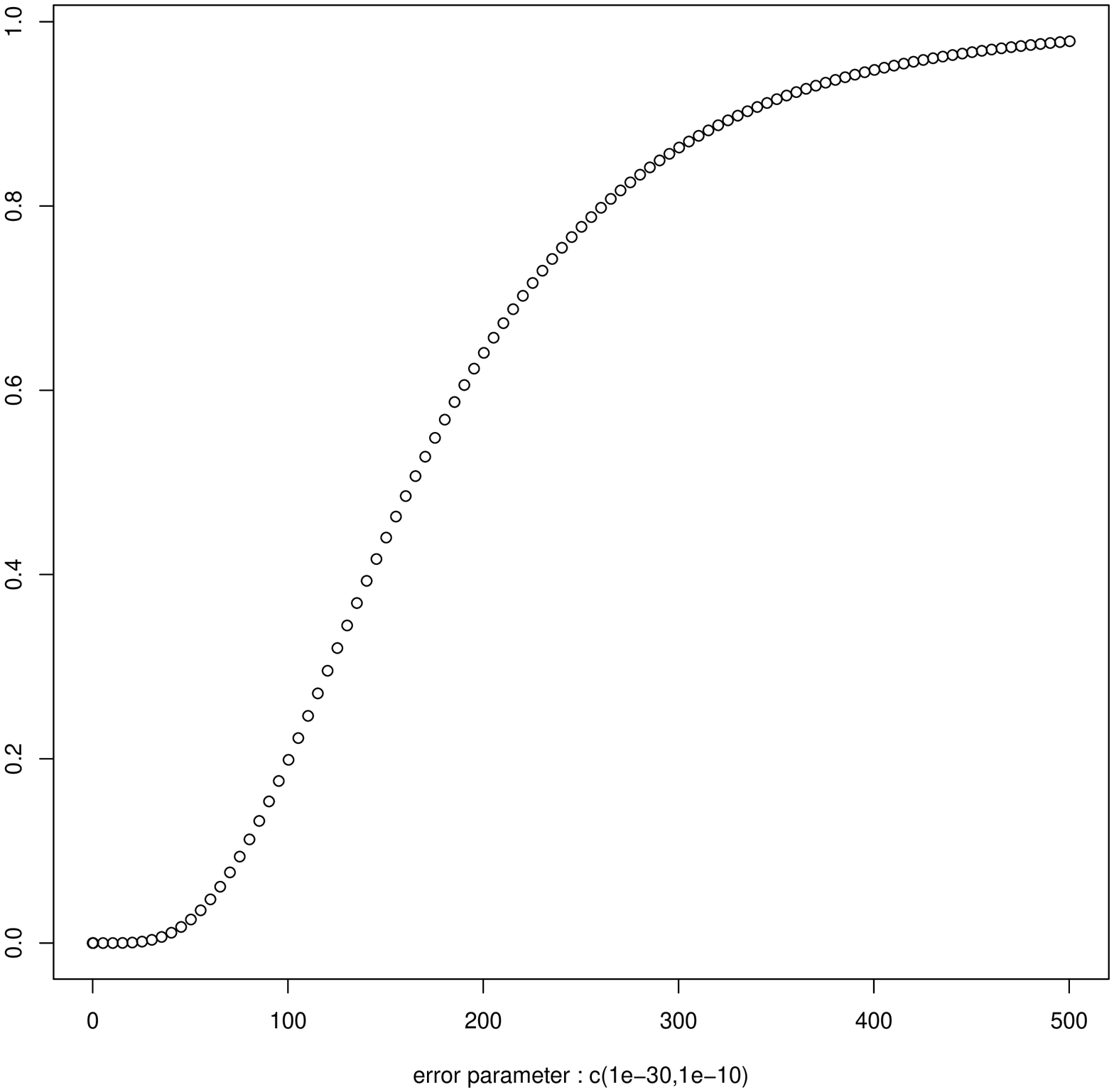}
\end{center}
\label{fig:1-20-300}
\caption{Mixture of a small eigenvalue and a large eigenvalue and effect of control parameter.}
\end{figure}
This example illustrates 
that inappropriate setting of the control parameter for the adaptive Runge-Kutta
method used in the hgm leads to a wrong answer.
The computation time is a few seconds for these examples, 
then we do not show detailed timing data.
\end{example}

The next example illustrates the behavior of our algorithm 
and an implementation for {\tt R} when 
the degrees of freedom becomes larger.
\begin{example}
We make the degrees of freedom $n_2$ to $200$.
\begin{verbatim}
plot(hgm.p2wishart(m=3,beta=c(1,20,300),q=50,
                   n1=10,n2=200,err=c(1e-30,1e-10),autoplot=1));
\end{verbatim}
This execution does not finish in a few seconds and takes 19 seconds
on Mac OS X 10.9 with 2.4GHz Intel Core i7 and 8G memory.
The output is the left graph of Figure \ref{fig:1-20-300-n2-200}.
A good choice of {\tt q0}, which is the initial evaluation point for the HGM
and of which default value is $0.3$,
improves the performance.
For example, the same evaluation with a different {\tt q0} as
\begin{verbatim}
plot(hgm.p2wishart(m=3,beta=c(1,20,300),q=300,
     n1=10,n2=200,err=c(1e-40,1e-10),autoplot=1,q0=0.1,verbose=1))
\end{verbatim}
finishes in 0.248 seconds.
\begin{figure}[tb]
\begin{center}
\includegraphics[width=6cm]{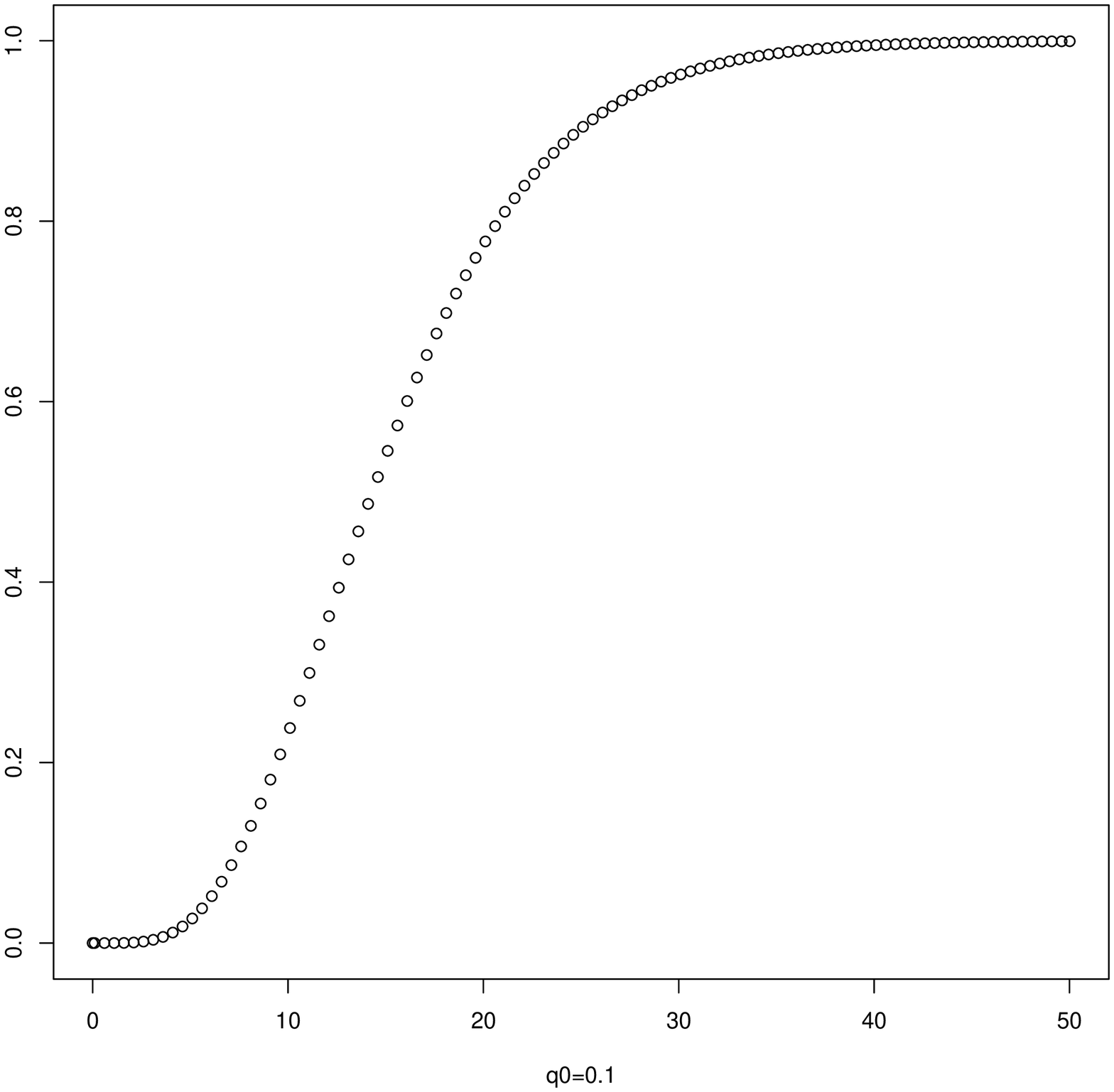} \quad
\includegraphics[width=6cm]{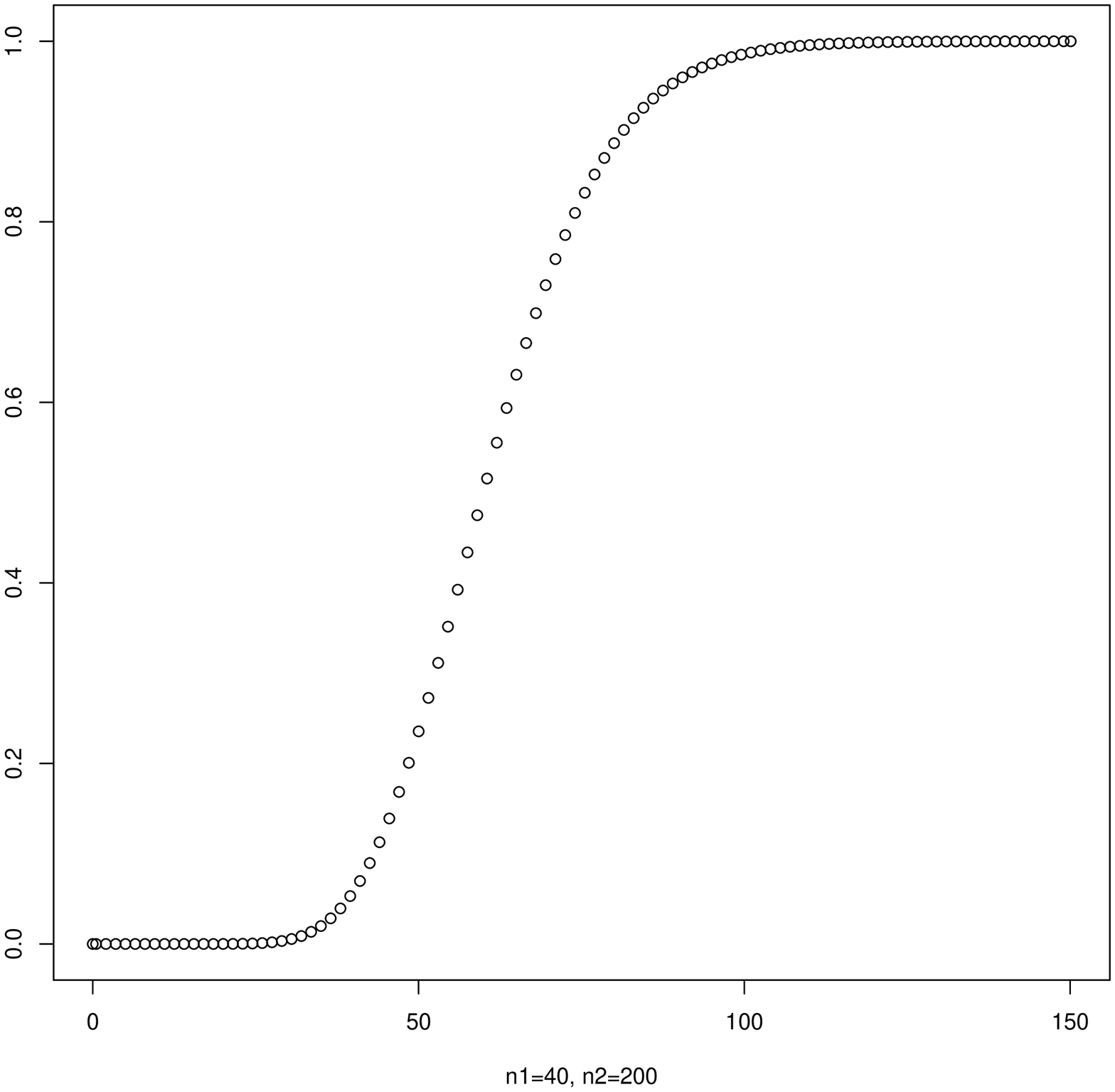}
\end{center}
\label{fig:1-20-300-n2-200}
\caption{Large degrees of freedom}
\end{figure}

We make the degrees of freedom $n_1$ to $300$.
\begin{verbatim}
plot(hgm.p2wishart(m=3,beta=c(1,20,300),q=300,q0=0.1,
                   n1=300,n2=200,err=c(1e-30,1e-10),autoplot=1));
\end{verbatim}
This execution stops with an error that the initial value
is zero, because the factor of ${}_2F_1$ in (\ref{eq:roys-maxroot-2sample02})
is too small.

Let us try other degrees of freedom $n_1=40$
with fixing the other parameters.
The output of
\begin{verbatim}
plot(hgm.p2wishart(m=3,beta=c(1,20,150),q=300,q0=0.1,
                   n1=40,n2=200,err=c(1e-60,1e-10),autoplot=1));
\end{verbatim}
is the right graph of Figure \ref{fig:1-20-300-n2-200}.
The execution takes 127.5 seconds.
\end{example}

\section{Discussion}
\label{sec:discussion}

Some open problems remain in this paper. 
We already mentioned  difficulties in proving that 
\eqref{eq:maxroot-null-constantine} and \eqref{eq:maxroot-null-venables} are equivalent.
Similarly differentiation of \eqref{eq:Kummer-2F1-02} should yield
\eqref{eq:l1-density-by-khatri}, but it does not seem to be obvious.

Another mathematically important question is the singularity of the differential operator
$g_i$ in \eqref{eq:2f1eq}.  
Note that $g_i$ has singularity in the diagonal region $x_i = x_j$, $i\neq j$.
Therefore HGM can not be used when there are multiple roots of $\Sigma_1 \Sigma_2^{-1}$.
The same problem was discussed in \citet{hashiguchi-etal-1f1} for the case of $\hyperF{1}{1}$.
In the case of $\hyperF{1}{1}$, Muirhead's differential operator $P_i$ annihilating $\hyperF{1}{1}$
is
\[
P_i = y_i \, \pdop_i^2 + \left\{ c - \frac{m-1}{2} -y_i + \2 \sum_{j=1, j \neq i}^m
\frac{y_i}{y_i - y_j} \right \} \pdop_i - 
\2  \sum_{j=1, j \neq i}^m \dfrac{y_j}{y_i - y_j}  \pdop_j 
-  a ,  \quad i=1,\dots,m.
\]
In \citet{hashiguchi-etal-1f1} it was conjectured that 
$y_i \prod_{j\neq i}(y_i -y _j)P_i$, $i=1,\dots,m$, generate a holonomic
left ideal in the Weyl algebra $D_m$. A proof for $m=2$ was given in Appendix A of
\citet{hashiguchi-etal-1f1}.
If this were the case for general $m$, 
then differential equations for the diagonal region could  be computed by
restriction algorithm for holonomic $D$-modules.
In \citet{hashiguchi-etal-1f1} restriction algorithm was
tried but failed for $m = 4$.
In fact \citet{kondo-2013}
proved hat the left $D$-ideal generated by $y_i \prod_{j\neq i}(y_i -y _j)P_i$
is not holonomic for $m = 4$.
On the other hand \citet{noro-2016} showed the use of 
l'H\^opital's rule in \citet{hashiguchi-etal-1f1}
can be generalized for computing a system of PDEs for various patterns of diagonalizations
of variables. By symbolic computations using Risa/Asir(\cite{asir}), it seems that the similar phenomenon occurs with
$\hyperF{2}{1}$.  However symbolic computations are heavier for 
$\hyperF{2}{1}$ than for $\hyperF{1}{1}$.  Hence further investigation is needed to clarify
the singularity of $g_i$ in the diagonal region $x_i = x_j$, $i\neq j$.

\appendix





\bibliographystyle{plainnat}
\bibliography{hgm2f1}

\begin{thebibliography}{17}
\providecommand{\natexlab}[1]{#1}
\providecommand{\url}[1]{\texttt{#1}}
\expandafter\ifx\csname urlstyle\endcsname\relax
  \providecommand{\doi}[1]{doi: #1}\else
  \providecommand{\doi}{doi: \begingroup \urlstyle{rm}\Url}\fi

\bibitem[Chikuse(1977)]{chikuse-1977}
Yasuko Chikuse.
\newblock Asymptotic expansions for the joint and marginal distributions of the
  latent roots of {$S\sb{1}S\sb{2}\sp{-1}$}.
\newblock \emph{Ann. Inst. Statist. Math.}, 29\penalty0 (2):\penalty0 221--233,
  1977.
\newblock ISSN 0020-3157.

\bibitem[Constantine(1963)]{constantine-1963}
A.~G. Constantine.
\newblock Some non-central distribution problems in multivariate analysis.
\newblock \emph{Ann. Math. Statist.}, 34:\penalty0 1270--1285, 1963.
\newblock ISSN 0003-4851.

\bibitem[Gupta and Nagar(2000)]{gupta-nagar-2000}
A.~K. Gupta and D.~K. Nagar.
\newblock \emph{Matrix {V}ariate {D}istributions}, volume 104 of \emph{Chapman
  \& Hall/CRC Monographs and Surveys in Pure and Applied Mathematics}.
\newblock Chapman \& Hall/CRC, Boca Raton, FL, 2000.
\newblock ISBN 1-58488-046-5.

\bibitem[Hashiguchi et~al.(2013)Hashiguchi, Numata, Takayama, and
  Takemura]{hashiguchi-etal-1f1}
Hiroki Hashiguchi, Yasuhide Numata, Nobuki Takayama, and Akimichi Takemura.
\newblock The holonomic gradient method for the distribution function of the
  largest root of a {W}ishart matrix.
\newblock \emph{J. Multivariate Anal.}, 117:\penalty0 296--312, 2013.
\newblock ISSN 0047-259X.

\bibitem[Herz(1955)]{herz-1955}
Carl~S. Herz.
\newblock Bessel functions of matrix argument.
\newblock \emph{Ann. of Math. (2)}, 61:\penalty0 474--523, 1955.
\newblock ISSN 0003-486X.

\bibitem[James(1964)]{james-1964ams}
Alan~T. James.
\newblock Distributions of matrix variates and latent roots derived from normal
  samples.
\newblock \emph{Ann. Math. Statist.}, 35:\penalty0 475--501, 1964.
\newblock ISSN 0003-4851.

\bibitem[Khatri(1967)]{khatri-1967ams}
C.~G. Khatri.
\newblock Some distribution problems connected with the characteristic roots of
  {$S\sb{1}S\sb{2}{}\sp{-1}$}.
\newblock \emph{Ann. Math. Statist.}, 38:\penalty0 944--948, 1967.
\newblock ISSN 0003-4851.

\bibitem[Khatri(1972)]{khatri-1972jmva}
C.~G. Khatri.
\newblock On the exact finite series distribution of the smallest or the
  largest root of matrices in three situations.
\newblock \emph{J. Multivariate Anal.}, 2:\penalty0 201--207, 1972.
\newblock ISSN 0047-259x.

\bibitem[Koev and Edelman(2006)]{koev-edelman}
Plamen Koev and Alan Edelman.
\newblock The efficient evaluation of the hypergeometric function of a matrix
  argument.
\newblock \emph{Math. Comp.}, 75\penalty0 (254):\penalty0 833--846, 2006.
\newblock ISSN 0025-5718.

\bibitem[Kondo(2013)]{kondo-2013}
Takayuki Kondo.
\newblock On a holonomic system of partial differential equations satisfied by
  the matrix ${}_1{F}_1$ (in japanese).
\newblock Master's thesis, Kobe universty, 2013.

\bibitem[Kuriki(1993)]{kuriki-1993as}
Satoshi Kuriki.
\newblock One-sided test for the equality of two covariance matrices.
\newblock \emph{Ann. Statist.}, 21\penalty0 (3):\penalty0 1379--1384, 1993.
\newblock ISSN 0090-5364.

\bibitem[Muirhead(1970)]{muirhead-1970}
R.~J. Muirhead.
\newblock Systems of partial differential equations for hypergeometric
  functions of matrix argument.
\newblock \emph{Ann. Math. Statist.}, 41:\penalty0 991--1001, 1970.
\newblock ISSN 0003-4851.

\bibitem[Muirhead(1982)]{muirhead-book}
Robb~J. Muirhead.
\newblock \emph{Aspects of {M}ultivariate {S}tatistical {T}heory}.
\newblock John Wiley \& Sons Inc., New York, 1982.
\newblock ISBN 0-471-09442-0.
\newblock Wiley Series in Probability and Mathematical Statistics.

\bibitem[Nakayama et~al.(2011)Nakayama, Nishiyama, Noro, Ohara, Sei, Takayama,
  and Takemura]{hgd}
Hiromasa Nakayama, Kenta Nishiyama, Masayuki Noro, Katsuyoshi Ohara, Tomonari
  Sei, Nobuki Takayama, and Akimichi Takemura.
\newblock Holonomic gradient descent and its application to the
  {F}isher-{B}ingham integral.
\newblock \emph{Advances in Applied Mathematics}, 47:\penalty0 639--658, 2011.
\newblock ISSN 0196-8858.

\bibitem[Noro(2016)]{noro-2016}
Masayuki Noro.
\newblock System of partial differential equations for the hypergeometric
  function ${}_1{F}_{1}$ of a matrix argument on diagonal regions.
\newblock In \emph{{ISSAC}'16 {P}roceedings of the {ACM} on {I}nternational
  {S}ymposium on {S}ymbolic and {A}lgebraic {C}omputation}, pages 381--388,
  2016.

\bibitem[RisaAsir developing team()]{asir}
RisaAsir developing team.
\newblock Risa/asir, a computer algebra system.
\newblock Available at http://www.math.kobe-u.ac.jp/Asir/.

\bibitem[Venables(1973)]{venables-1973jmva}
W.~Venables.
\newblock Computation of the null distribution of the largest of smallest
  latent roots of a beta matrix.
\newblock \emph{J. Multivariate Anal.}, 3:\penalty0 125--131, 1973.
\newblock ISSN 0047-259x.

\end{thebibliography}

\end{document}